\documentclass{amsart}
\usepackage{amssymb,amsmath,enumerate}

\theoremstyle{plain}
\newtheorem{theorem}{Theorem}[section]
\newtheorem{proposition}[theorem]{Proposition}
\newtheorem{lemma}[theorem]{Lemma}

\theoremstyle{definition}
\newtheorem{definition}[theorem]{Definition}
\newtheorem{remark}[theorem]{Remark}

\theoremstyle{plain}

\newcommand{\comment}[1]{}
\newcommand{\ol}{\ensuremath{\overline}}
\newcommand{\ora}{\ensuremath{\overrightarrow}}
\newcommand{\what}{\ensuremath{\widehat}}
\newcommand{\wtilde}{\ensuremath{\widetilde}}
\newcommand{\ve}{\ensuremath{\varepsilon}}
\newcommand{\B}{\ensuremath{\mathbb{B}}}
\newcommand{\R}{\ensuremath{\mathbb{R}}}
\newcommand{\U}{\ensuremath{\mathbb{U}}}
\newcommand{\cB}{\ensuremath{\mathcal{B}}}
\newcommand{\cP}{\ensuremath{\mathcal{P}}}
\newcommand{\cM}{\ensuremath{\mathcal{M}}}
\newcommand{\cU}{\ensuremath{\mathcal{U}}}
\newcommand{\cA}{\ensuremath{\mathcal{A}}}
\newcommand{\cO}{\ensuremath{\mathcal{O}}}
\newcommand{\cX}{\ensuremath{\mathcal{X}}}
\newcommand{\pF}{\ensuremath{\partial_F}}
\newcommand{\pL}{\ensuremath{\partial_L}}
\newcommand{\weakly}{\ensuremath{\stackrel{w}{\longrightarrow}}}

\begin{document}

\title[Maximum Principle on Infinite Dimensional Manifolds]{Pontryagin Maximum Principle for Control Systems on Infinite Dimensional Manifolds}

\author[Kipka]{Robert J. Kipka}
\address[Kipka]{Department of Mathematics, Western Michigan University, Kalamazoo, MI, 49008-5248}
\email{robert.j.kipka@wmich.edu}

\author[Ledyaev]{Yuri S. Ledyaev}
\address[Ledyaev]{Department of Mathematics, Western Michigan University, Kalamazoo, MI, 49008-5248}
\email{iouri.lediaev@wmich.edu}





\maketitle

\begin{abstract}
 We discuss a mathematical framework for analysis of optimal control
problems on infinite-dimensional manifolds. Such problems arise in study
of optimization for  partial differential equations with some  symmetry.
It is shown that some nonsmooth analysis methods and Lagrangian charts techniques can be used for study of global variations of optimal trajectories of
such control systems and derivation of maximum principle for them.
\end{abstract}

\section{Introduction}

In this paper we study control problems in which the state evolves on an
infinite-dimensional  manifold $M$ which is  modeled over  Banach space $E$.
The problem of optimal control with which we will concern ourselves is the following Mayer problem:

\textbf{Problem ($\mathcal{P}$):} \emph{Minimize $\ell(q(0), q(T))$ subject to dynamic constraint}
\begin{equation}
\label{eq:dynamic-constraint}
\dot{q}(t) = f(t, q(t), u(t)), \quad q(0)=q_0
\end{equation}
\emph{and endpoint constraints}
\begin{equation}
\label{eq:state-constraint}
\left(q(0), q(T)\right) \in S \subset M \times M.
\end{equation}
where $q(t)$ describes a state of the control system \eqref{eq:dynamic-constraint}, $u(t)$ is a control
function taking values in a set $\U$, $f(t,q,u)$ is a parametric family
of vector fields $f: [0,T]\times M\times \U \to TM$ describing a dynamics of the system, the set $S$ describes
end constraints for trajectory of \eqref{eq:dynamic-constraint} and the function $\ell$ is a cost functional.

Study of infinite-dimensional optimal control problems  in which a state vector
$q(t)$ moves in a linear Banach space $E$ have a long history and it continues
to attract  attention due to numerous applications (see
\cite{fattorini1999infinite,krastanov2011pontryagin,troltzsch2010optimal}
for additional references).

Surprisingly, theory of optimal control on infinite-dimensional manifolds
is not developed to the same degree of completeness as in a linear space case. But such optimal control problems can appear naturally from  conservation laws which are due to some internal symmetry of the problem.
We should mention also that even in the case of finite-dimensional manifolds
known optimal control techniques lack some tools which demonstrated their effectiveness
in finite-dimensional linear spaces.

One of such tools is nonsmooth analysis \cite{clarke1983,clarkeand1998,clarke2013} which have been introduced originally
for study of optimal control problems as nonstandard nonsmooth problems of calculus of variations. Even if the original focus was on non-differentiable (nonsmooth) problems  it was demonstrated later that nonsmooth analysis
methods can be successfully used for study of problems with smooth data.

This paper contains a mathematical framework for study of smooth
optimal control problems on infinite-dimensional manifolds by using
nonsmooth analysis techniques, in particular a
new Subbotin-like formula relating weak Dini derivatives and limit subgradients
of Lipschitz function.  We demonstrate how these tools
can be used for obtaining sufficient conditions for metric regularity
of constraints, existence of exact nonsmooth penalty function
for optimal control problems on manifolds. Such exact penalty
functions are exploited for derivation of normal optimality
conditions in the form of  Pontryagin maximum principle.

\subsection{Main Definitions and Assumptions}

We provide here the precise definitions and assumptions used for problem ($\mathcal{P}$). The set $\U$ is assumed to be a complete, separable metric space. We say that $u : \left[0,T\right] \rightarrow \U$ is \emph{measurable} if it is the pointwise a.e. limit of continuous functions and we suppose that $u : \left[0,T\right] \rightarrow \U$ is measurable. A \emph{control system} is a mapping $f : \left[0,T\right] \times M \times \U \rightarrow TM$ which satisfies $\pi \circ f(t,q,u) = q$ for all $(t,q,u)$, where the map $\pi : TM \rightarrow M$ is the projection map sending $v \in T_qM$ to its base-point $q$.

Function $\ell : M \times M \rightarrow \R$ is assumed locally Lipschitz and $S \subset M \times M$ closed. A mapping $x : \left[0,T\right] \rightarrow E$ is \emph{absolutely continuous} if there is a Bochner integrable function $v$ such that
\begin{equation*}
  x(t_2) - x(t_1) = \int_{t_1}^{t_2} v(s) \,ds
\end{equation*}
for all $t_1, t_2 \in \left[0,T\right]$. Mapping $q : \left[0,T\right] \rightarrow M$ is absolutely continuous if it is absolutely continuous in local coordinates and we use the notation $q(t;q_0, u)$ for the absolutely solution to \eqref{eq:dynamic-constraint} satisfying $q(0;q_0, u) = q_0$. For a detailed introduction to Bochner integral we suggest \cite{diestel-uhl1977}.

In order to state precise assumptions for control system $f$ we recall that if $(\cO, \varphi)$ is a coordinate chart, then $\varphi_*f : \left[0,T\right] \times \varphi(\cO) \times \U \rightarrow E$ is the \emph{local coordinate expression} for $f$, defined by $(\varphi_* f)(t,x,u):=\varphi_*(\varphi^{-1}(x))f(t, \varphi^{-1}(x),u)$.

\textbf{Assumption ($D$):} Control system $f$ \emph{satisfies Assumption ($D$)} along a continuous map $q_* : \left[0,T\right] \rightarrow M$  if there are finitely many coordinate charts $(\cO_i, \varphi_i)$ and functions $m, k \in L^1$ such that when $q_*(t) \in \cO_i$ we have for almost all $t$, for all $u \in \U$,
\begin{equation}
\label{eq:integrable-bdd}
\begin{aligned}
\sup_{x \in \varphi_i(\cO_i)}  \left\|(\varphi_{i \, *}f)(t, x, u) \right\|_E &\le m(t) \\ \sup_{x \in \varphi_i(\cO_i)}  \left\|(\varphi_{i \, *}f)_x(t, x, u) \right\|_{L(E,E)} &\le m(t)
\end{aligned}
\end{equation}
and
\begin{equation}
\label{eq:integrable-lip}
\begin{aligned}
\sup_{x,y \in \varphi_i(\cO_i)}  \left\|(\varphi_{i \, *}f)(t, x, u) - (\varphi_{i \, *}f)(t,y,u) \right\|_E \le k(t) \left\|x - y\right\|_E \\
\sup_{x,y \in \varphi_i(\cO_i)}  \left\|(\varphi_{i \, *}f)_x(t, x, u) - (\varphi_{i \, *}f)_x(t,y,u) \right\|_{L(E,E)} \le k(t) \left\|x - y\right\|_E
\end{aligned}
\end{equation}
We assume throughout the paper that $f$ satisfies this assumption.

\textbf{Banach Space Assumption:} The space $E$ over which $M$ is modelled is assumed to be a reflexive space. If $M \ne E$ then we also make the assumption that $E$ admits a $C^2$-smooth bump function with locally Lipschitz second derivative. Note that any Hilbert space meets these criteria, as do the Sobolev spaces $W^{k,p}$ ($3 \le p < \infty$) \cite{adams1975,bonic1966smooth}, which are of interest in the study of partial differential equations.

\begin{remark}
  The smoothness assumption is required only for our use of Lagrangian charts. When $M = E$ there is no need for such charts and our results then hold under the assumption that $E$ is reflexive.
\end{remark}

\subsection{Nonsmooth Analysis}

Even when all of the data for Problem ($\cP$) are smooth, techniques of nonsmooth analysis prove themselves to be a useful for the study of control problems. In this section we define the main objects of nonsmooth analysis needed for this paper. Following \cite{ledyaevzhu2007}, we make the following definitions for lower semicontinuous function $\psi : M \rightarrow \R \cup \left\{+\infty\right\}$.
\begin{definition}
  If there exists a $C^1$-smooth function $g : M \rightarrow \R$ such that $\psi - g$ attains a local minimum at $q$ and $\zeta = dg_q$, we say that $\zeta$ is a \emph{Fr\'echet subgradient} for $\psi$. We write $\pF \psi (q) \subset T_q^*M$ for the (possibly empty) set of such covectors.
\end{definition}

\begin{definition}
  If there exist sequences $q_k \rightarrow q$ such that $\psi(q_k) \rightarrow \psi(q)$ and $\zeta_k \in \pF \psi(q_k)$ such that $\zeta_k \weakly \zeta$ then we say that $\zeta$ is a \emph{limiting subgradient} for $\psi$. We write $\pL \psi(q) \subset T_q^*M$ for the set of all such covectors.
\end{definition}

\begin{definition}
If $C \subset M$ is a closed set, we define $N_C^F(q)$ (resp. $N_C^L(q)$) to be $\pF I_C(q)$ (resp. $\pL I_C(q)$), where $I_C : M \rightarrow \R \cup \left\{+\infty\right\}$ is the lower semicontinuous function defined to be zero on $C$ and $+\infty$ elsewhere.
\end{definition}

Subdifferentials are coordinate-free in the sense that if $\varphi : \cO \rightarrow E$ is a coordinate chart then
\begin{equation}
\label{eq:invariance}
  \varphi^* \pL \left(\ell \circ \varphi^{-1}\right)(x) = \pL \ell(q),
\end{equation}
where $\varphi(q) = x$.

Normals to sets are of particular importance to us and the following property will be of use in proving nondegeneracy of the adjoint arc:
\begin{definition}
  \label{defn:snc}
  A closed set $C \subset E$ is \emph{sequentially normally compact} if for any sequences $c_k \in C$ and $\zeta_k \in N_C^F(c_k)$ with $c_k \rightarrow c$ and $\zeta_k \weakly \zeta$ we have $\left\|\zeta\right\|_{E^*} = 0$ if and only if $\left\|\zeta_k \right\|_E \rightarrow 0$.
\end{definition}
This property is well-understood for a large class of sets and \cite{mordu2006} is a good starting point for further study. It is worth noting if $C$ is defined through $C^1$-smooth constraints
\begin{equation*}
\begin{aligned}
  g_i(x) & \le 0 & 1 & \le i \le r  \\
  h_j(x) & = 0 & 1 & \le j \le s
\end{aligned}
\end{equation*}
then $C$ is sequentially normally compact at $c$ if
\begin{equation*}
  \sum_{i = 1}^r \alpha_i \nabla g_i(c) + \sum_{j = 1}^s \beta_j \nabla h_j(c) \ne 0
\end{equation*}
whenever $\alpha_i \ge 0$ and $\sum_{i = 1}^r \alpha_i + \sum_{j = 1}^s \left|\beta_j\right| = 1$. On the other hand, a singleton set $\left\{c\right\}$ is sequentially normally compact if and only if $E$ has finite dimension.

When $E$ is reflexive and $C$ is sequentially normally compact, subdifferentials of the distance function have the following properties \cite{mordu2006}:
\begin{proposition}
  If $x \not \in C$ and $\zeta \in \pL d_C(x)$, then $\left\|\zeta\right\|_{E^*} = 1$ and there is a point $c \in C$ satisfying $\left\|x - c\right\|_E = d_C(x)$ and $\zeta\in N_C^L(c)$. On the other hand, if $x \in C$, then we have $\pL d_C(x) \subset N_C^L(x)$.
\end{proposition}

\section{Statement of Main Results}

We write $\cU$ for the set of measurable controls and define the Pontryagin function $H : \left[0,T\right] \times T^*M \times \U \rightarrow \R$ by
\begin{equation*}
  H(t,\zeta, u) := \left<\zeta, f(t, q, u)\right>,
\end{equation*}
for $\zeta \in T_q^*M$. The vector field $\ora{H}(t, \zeta,u)$ is the \emph{Hamiltonian lift} of $H$, which is defined through the natural symplectic structure on $T^*M$. We suggest \cite{lang1999} for an introduction in the infinite-dimensional case. 
With these constructs in mind, we state the main results of this paper.

\begin{theorem}
\label{thm:pmp}
  Suppose that a pair $(q_0, u_*) \in M \times \cU$ is optimal for Problem ($\cP$) with trajectory $q_*(t):=q(t;q_0, u_*)$. There exists $\lambda_0 \in \left\{0,1\right\}$ and an absolutely continuous map $\zeta : \left[0,T\right] \rightarrow T^*M$ determined by
  \begin{equation}
    \label{eq:adjoints}
    \dot{\zeta}(t) = \ora{H}(t, \zeta(t), u_*(t))
  \end{equation}
  and boundary conditions $\left(\zeta(0), - \zeta(T)\right) \in  \lambda_0 \pL \ell(q_*(0), q_*(T)) + N_S^L(q_*(0), q_*(T))$ which satisfies the maximum principle:
  \begin{equation}
    \label{eq:max-princ}
H(t, \zeta(t), u_*(t)) = \max_{u \in \U} H(t, \zeta(t), u) \quad a.a. \; t.
  \end{equation}
  Further, $\left(\lambda_0, \zeta(t) \right) \ne 0$ for all $t$.
\end{theorem}

We also introduce a notion of \emph{weak controllability in the direction of $S$} (Definition \ref{defn:weakly-controllable} of Section \ref{section:metric-regularity}) and we prove the following exact penalization type result:
\begin{theorem}
\label{thm:penalization}
  If the pair $(q_0, u_*)$ optimal for Problem ($\cP$) and is weakly controllable in the direction of $S$ then there exists $\kappa > 0$ and such that $(q_0, u_*)$ is optimal for Problem ($\cP_{\mathrm{penalty}}$) whose cost is given by
  \begin{equation*}
    J(q,u) := \ell(q,q(t;q,u)) + \kappa \Phi(q,q(t;q,u))
  \end{equation*}
  where $\Phi : M \times M \rightarrow \R$ is a locally defined Lipschitz function satisfying $\pL \Phi(q, q^\prime) \subset N_S^L(q,q^\prime)$ when $(q, q^\prime) \in S$.
\end{theorem}
In Theorem \ref{thm:penalization} the constraint \eqref{eq:state-constraint} has been removed through an exact penalization technique.

The proof of Theorem \ref{thm:penalization} relies on a metric regularity argument and we next introduce the essential tools used for this argument. Following this, we provide a proof of Theorem \ref{thm:penalization} and conclude the paper with a proof of Theorem \ref{thm:pmp}. A central technique throughout this paper is study of global variations of trajectories through Lagrangian charts and we develop these techniques next.

\section{Global Variation of Trajectories and Pseudometric Space}

We wish to study variations of a trajectory $q(t;q_0, u)$ in which both the initial condition $q_0$ and the control are varied. It is useful to employ sliding mode type variations, which are based on relaxed controls,
a technique which can be traced back to the work of L.C. Young \cite{young1937} in the Calculus of Variations and Gamkrelidze and Warga \cite{gamk1978,warga1962,warga1972} in optimal control.

Given $u_0 \in \U$, define a measure $\delta_{u_0}$ on $\U$ by setting
\begin{equation*}
  \int_{\U} h(u) d\delta_{u_0} := h(u_0)
\end{equation*}
for an arbitrary function $h : \U \rightarrow E$.

\begin{definition}
We say that $\nu$ is a \emph{relaxed control} if for some integer $n$ there are $n$ controls $u_i \in \cU$ and convex coefficients $\lambda_1, \dots, \lambda_n$ such that $\nu(t) = \sum_{i = 1}^n \lambda_i \delta_{u_i(t)}$. We write $\cM$ for the set of relaxed controls.
\end{definition}

Control system $f$ is extended to handle relaxed controls by setting
\begin{equation*}
  \what{f}(t, q, \nu(t)) : = \int_{\U} f(t, q, u) \, d\nu(t) = \sum_{i = 1}^n \lambda_i f(t, q, u_i(t)).
\end{equation*}

For problems with constraint \eqref{eq:state-constraint} an optimal control $u_*$ may not be optimal among relaxed controls, a phenomenon known as \emph{relaxation gap}. Nonetheless, relaxed trajectories can be approximated using usual controls and so are useful for global variations of trajectories. In order to give quantitative estimates related to such variations, employ the technique of Lagrangian charts, introduced in \cite{kl2014}.

\subsection{Lagrangian Charts}

Fix an absolutely continuous map $q_* : \left[0,T\right] \rightarrow M$. We would like to study continuous curves which are close to $q_*$. By the smoothness assumption on $E$, we can construct a $C^2$-smooth, nonautomous vector field $V_t$ with locally Lipschitz second derivative such that $V_t$ extends $\dot{q}_*$ in the sense that for almost all $t$ we have $V_t(q_*(t)) = \dot{q}_*(t)$. Let $P_{s,t}$ denote the flow of $V_t$. Choose a coordinate chart $(\cO, \varphi)$ with $q_*(0) \in \cO$ and let $\psi_t : P_{0,t}(\cO) \rightarrow \varphi(\cO)$ be the map $\psi_t := \varphi \circ P_{t, 0}$. We refer to $\left(P_{0,t}(\cO), \psi_t\right)_{t \in \left[0,T\right]}$ as a \emph{Lagrangian chart}, in analogy with the concept of Lagrangian coordinates from fluid dynamics.

Given a pair $(q,u)$ such that $q(t;q,u) \in P_{0,t}(\cO)$ for all $t$ the curve $x(t):=\psi_t(q(t;q,u))$ satisfies
\begin{equation}
\label{eq:dotx}
\dot{x}(t) := \left(\psi_{t \, *} f\right)(t, x(t), u(t))  - \left(\psi_{t \, *} V_t\right)(x(t)) \quad a.a. \, t.
\end{equation}
A proof of \eqref{eq:dotx} in the context of Banach manifolds can be found in \cite{kl2014extension}. With \eqref{eq:dotx} in mind we introduce a control system $g : \left[0,T\right] \times \varphi(\cO) \times \U \rightarrow E$ defined through
\begin{equation}
\label{eq:defn-g}
  g(t,x,u) := \left(\psi_{t \, *}f \right)(t,x,u) - \left(\psi_{t \, *} V_t\right)(x).
\end{equation}
Under our assumptions on $f$ and $V_t$, control system $g$ satisfies Assumption ($D$) along constant trajectory $\varphi(q_*(0))$. We write $m_g$ and $k_g$ for the functions in \eqref{eq:integrable-bdd} and \eqref{eq:integrable-lip}, in which we may take $\varphi_i := Id_E$.
Given $x_0 \in \varphi(\cO)$ and $u \in \U$ we will write $x(t;x_0,u)$ for the solution to
\begin{equation*}
\dot{x}(t;x_0,u) = g(t, x(t;x_0,u), u(t))
\end{equation*}
which satisfies $x(0;x_0, u) = x_0$.

\subsection{Pseudometric Space and Approximation} \label{pseudometric-subsection}

The Lagrangian chart $\left(P_{0,t}(\cO), \psi_t\right)_{t \in \left[0,T\right]}$ provides a way to quantify differences between trajectories which are close to $q_*$. We next introduce a pseudometric structure on a set of controls generating such trajectories.

Introduce a neighborhood $\cO_0$ of $q_*(0)$ for which $\ol{\cO}_0 \subset \cO$ and let $\cX$ denote the set of pairs $(q_0, u)$ such that $q(t;q_0,u) \in P_{0,t}(\ol{\cO}_0)$ for all $t$. Equivalently, $\cX$ can be thought of as the set of all $(x_0, u)$ such that $x(t;x_0, u) \in \varphi(\ol{\cO}_0)$ for all $t$. We write $\cX^\circ$ for the smaller set of pairs $(q_0, u) \in \cX$ for which $q(t;q_0, u) \in P_{0,t}(\cO_0)$ for all $t$.

We introduce the following pseudometric on $\cX$:
\begin{equation*}
\begin{aligned}
  & \rho(q_1,u_1; q_2, u_2) := \left\|x_1 - x_2\right\|_E \\
& \hspace{1cm} + \int_0^T \left\|g(t,x(t;x_1, u_1), u_1(t)) - g(t, x(t;x_2, u_2), u_2(t)) \right\|_E \, dt \\
 & \hspace{2cm} \int_0^T \left\|g_x(t,x(t;x_1, u_1), u_1(t)) - g_x(t, x(t;x_2, u_2), u_2(t)) \right\|_{L(E,E)} \, dt,
\end{aligned}
\end{equation*}
where $x_i := \varphi(q_i)$.

We develop several approximation results for global variations in the space $\cX$. All are based on the following lemma, which in turn mirrors the technique of chattering control found in \cite{berkovitz1974optimal,gamk1978}. A complete proof of Lemma \ref{lem:chattering} is given in \cite{kl2014} for the case $E = \R^n$.
\begin{lemma}
\label{lem:chattering}
Let $F$ be a Banach space and $h_i : \left[0,T\right] \rightarrow F$ Bochner integrable functions, $i = 1, \dots, n$. For any convex coefficients $\lambda_1, \dots, \lambda_n$ and $\ve > 0$ there exist disjoint measurable sets $A_i \subset \left[0,T\right]$ with $m(A_i) = \lambda_i T$ such that
\begin{equation}
\label{eq:chattering-bound}
  \left\|\int_{\left[0,t\right] \cap A_i} h_i(s) \, ds - \lambda_i \int_0^t h_i(s) \, ds \right\|_F < \ve \quad i = 1, \dots, n
\end{equation}
\end{lemma}

\begin{proof}
We provide a sketch, referring the reader to \cite{kl2014} for additional detail. First check that suffices to assume that function $h_i$ are continuous. Then form a partition $0 := t_0 < t_1 < \dots < t_{r-1} < t_r := T$. For $j = 0, \dots, r$ let $I_j := \left[t_j, t_{j+1}\right]$ and define
\begin{equation*}
  I_{j, k} := \left[t_j + \sum_{i = 1}^{k-1} \lambda_i \, \left(t_{j+1}- t_j\right), t_j + \sum_{i = 1}^{k} \lambda_i \left(t_{j+1}- t_j\right) \right].
\end{equation*}
Let $A_k = \cup_{j = 0}^{r-1} I_{j, k}$. A careful estimate of the left-hand side in \eqref{eq:chattering-bound} shows that it goes to zero with the diameter of the partition.
\end{proof}

A first application of Lemma \ref{lem:chattering} is that relaxed trajectories may be approximated uniformly with usual trajectories.
\begin{lemma}
\label{lem:approx-relax}
Suppose that $q_0 \in M$ and $\nu \in \cM$ are such that $q(t;q_0,\nu) \in P_{0,t}(\cO_0)$ for all $t$ and let $\ve > 0$ be given. There is a control $u \in \cU$ such that $(q_0, u) \in \cX^\circ$ and
\begin{equation}
\label{eq:approx-relax}
  \max_{t \in \left[0,T\right]} \left\| \psi_t(q(t;q_0,\nu)) - \psi_t(q(t;q_0,u)) \right\|_E < \ve.
\end{equation}
\end{lemma}
\begin{proof}
Set $x_0 := \varphi(q_0)$ and check that it suffices to prove that if $\nu \in \cM$ is such that $x(t;x_0,\nu) \in \varphi(\cO_0)$ for all $t$ then there is a control $u \in \cU$ such that
\begin{equation*}
  \max_{t \in \left[0,T\right]} \left\| x(t;x_0,\nu) - x(t;x_0,u) \right\|_E < \ve  .
\end{equation*}
Choose controls $u_i$ and convex coefficients $\lambda_1, \dots, \lambda_n$ such that $\nu(t) := \sum_{i = 1}^n \lambda_i \delta_{u_i(t)}$ and consider the functions
\begin{equation*}
  t \mapsto \what{g}(t,x(t;x_0,\nu), \delta_{u_i(t)}).
\end{equation*}
Given a measurable set $A \subset \left[0,T\right]$ let $\chi_{A}$ be the function $\chi_A(t) = 1$ for $t \in A_i$ and $\chi_A(t)=0$ otherwise. By Lemma \ref{lem:chattering} we can choose disjoint measurable sets $A_i \subset \left[0,T\right]$ with $m(A_i) = \lambda_i T$ such that control $w(t) = \sum_{i = 1}^n\chi_{A_i}(t) u_i(t)$
satisfies
\begin{equation*}
  \max_{t \in \left[0,T\right]} \left\|  \int_0^t \what{g}(s,x(s;x_0,\nu),\nu(s)) \, ds -  \int_0^t \what{g}(s,x(s;x_0,\nu), \delta_{w(s)}) \, ds \right\|_E < \ve.
\end{equation*}
For $t \in \left[0,T\right]$ we now have
\begin{equation*}
  \begin{aligned}
    & \left\|x(t;x_0, \nu) - x(t;x_0, w) \right\|_E \le \ve T+ \int_0^t \left\|g(s,x(s;x_0,\nu), w(s)) - g(s,x(s;x_0,w), w(s)) \right\|_E \, ds \\
  &  \hspace{1cm} \le \ve T+ \int_0^t k_g(s)  \left\|x(s;x_0, \nu) - x(s;x_0, w) \right\|_E  \, ds.
  \end{aligned}
\end{equation*}
An application of the Gronwall lemma completes the proof.
\end{proof}

\subsection{Global Variation of Trajectories}

We now turn to our study of global variations.
Let $(q_0, u) \in \cX$ be given. We consider a global variation of $q(t;q_0, u)$ corresponding to relaxed control $\nu$ and tangent vector $v_0 \in T_{q_0}M$ satisfying $\left\|\varphi_*(q_0)v_0\right\|_E \le 1$. For $q \in \cO$, we let $\cB_q\subset T_qM$ denote the set of such tangent vectors $v_0$.

For relaxed control $\nu$ and tangent vector $v_0 \in \cB_{q_0}$ we define
\begin{equation}
\label{eq:initial-variation}
  q^\lambda_0 := \varphi^{-1}\left(\varphi(q_0) + \lambda \varphi_*(q_0) v_0\right)
\end{equation}
and we consider a global variation of $q(t;q_0, u)$ defined by
\begin{equation}
\label{eq:basic-variation}
q^\lambda(t) := q(t;q^\lambda_0, (1 - \lambda) \delta_u + \lambda \nu).
\end{equation}
The following lemma assures us that such a variation approaches $q^0(t)$ at a uniform linear rate:
\begin{lemma}
\label{lem:sliding-bound}
There is a constant $c_0$ such that for any $(q_0, u) \in \cX^\circ$ the variation $q^\lambda(t)$ defined by \eqref{eq:basic-variation} satisfies
\begin{equation}
\label{eq:sliding-bound}
  \left\|\psi_t(q^\lambda(t)) - \psi_t(q^0(t)) \right\|_E < c_0 \lambda
\end{equation}
when $\lambda$ is sufficiently small.
\end{lemma}

\begin{proof}
Define $x^\lambda(t) = \psi_t(q^\lambda(t))$ and note that if $x^\lambda_0 := \varphi(q_0) + \lambda \varphi_*(q_0) v_0$ then $x^\lambda(t) = x(t;x^\lambda_0,  (1 - \lambda) \delta_u + \lambda \nu)$.
For $t \in \left[0,T\right]$ we have
\begin{equation*}
  \begin{aligned}
&    \left\|x^\lambda(t) - x^0(t) \right\|_E \le \lambda + \int_0^t \left\|\what{g}(s, x^\lambda(s), \delta_{u(s)}) - \what{g}(s, x^0(s), \delta_{u(s)}) \right\|_E \, ds + 2 \lambda \left\|m_g\right\|_{L^1} \\
& \hspace{1cm} \le \lambda + \int_0^t k_g(s) \left\|x^\lambda(s)- x^0(s) \right\|_E \, ds + 2 \lambda \left\|m_g\right\|_{L^1}.
  \end{aligned}
\end{equation*}
The result now follows from the Gronwall lemma.
\end{proof}

As an important consequence, we may approximate relaxed variation $q^\lambda$ using a carefully constructed ``usual control.'' Moreover, this approximation is compatible with pseudometric $\rho$. The following Lemma can be established using an application of Gronwall lemma, along with Lemma \ref{lem:chattering}.
\begin{lemma}
\label{lem:pseudo-approx}
There exist constants $c_1, c_2$ such that for any pair $(q_0, u) \in \cX^\circ$, for any relaxed control $\nu \in \cM$, $v_0 \in \cB_{q_0}$, $\ve > 0$, and $\lambda> 0$ sufficiently small there is a control $w^\lambda \in \cU$ such that
\begin{equation}
\label{eq:lambda-approx}
  \max_{t \in \left[0,T\right]} \left\|\psi_t(q^\lambda(t)) - \psi_t(q(t;q^\lambda_0, w^\lambda)) \right\|_E < \ve,
\end{equation}
where $q^\lambda_0$ and $q^\lambda(t)$ are defined by \eqref{eq:initial-variation} and \eqref{eq:basic-variation}, respectively.
In addition,
\begin{equation}
\label{eq:rho-bound}
  \rho(q_0^\lambda,w^\lambda;q_0, u) < c_1 \lambda + c_2 \ve.
\end{equation}
\end{lemma}

Thus, we may vary any trajectory corresponding to $(q_0, u) \in \cX^\circ$ using \eqref{eq:basic-variation} and the resulting variation can be approximated nicely using controls $w^\lambda \in \cU$.

\section{Set-Constrained Lower Derivative and the Free Endpoints Problem}

In this section we prove the Pontryagin Maximum Principle under the assumption that $S = M \times M$. The following proposition gives us a useful characterization of the adjoint arc. Let $Q_{s,t}$ be the flow of a $C^1$-smooth vector field $X_t$ with locally integrable Lipschitz first derivative and let $H_X : \left[0,T\right] \times T^*M\rightarrow \R$ be defined by $H_X(t,\zeta):= \left<\zeta, X_t(q)\right>$.
\begin{proposition}
\label{prop:pull-back-flow}
If $\eta \in T_{Q_{0,T}(q_0)}^*M$ then $\zeta(t) := -Q_{t,T}^*\eta$ if and only if $\zeta$ is a solution to
\begin{equation*}
  \dot{\zeta}(t) = \ora{H}_X(t,\zeta) \quad a.a. \; t
\end{equation*}
with $\zeta(T) = - \eta$. In particular, if $Q_{s,t}$ is the flow of the vector field $(t,q) \mapsto f(t, q, u(t))$ then $\zeta$ satisfies \eqref{eq:adjoints}.
\end{proposition}

\begin{proof}A proof for the case in which $M$ is modeled over $\R^n$ can be found in \cite{kl2014}. This proof can be easily generalized to the Banach space setting.
\end{proof}

We also have the following, whose proof we leave to the reader.
\begin{proposition}
\label{prop:int-pmp}
  If absolutely continuous mapping $\zeta : \left[0,T\right] \rightarrow T^*M$ and control $u_*$ satisfy
  \begin{equation}
  \label{eq:int-pmp}
    \inf_{u \in \cU} \int_0^T H(t, \zeta(t), u_*(t)) - H(t, \zeta(t), u(t)) \, dt \ge 0
  \end{equation}
  then \eqref{eq:max-princ} holds pointwise almost everywhere.
\end{proposition}

The basic idea for the proof of the maximum principle is the following. Find $(\xi, \eta) \in \pL \ell(x_*(0), x_*(T))$ such that $-Q_{0,T}^*\eta = \xi$, where $Q_{s,t}$ is the flow of $(t, q) \mapsto f(t,q,u_*(t))$. Define $\zeta(t) = -Q_{t, T}^* \eta$, so that $\left(\zeta(0), -\zeta(T)\right) \in \pL \ell(x_*(0), x_*(T))$ and verify that $\zeta$ can be chosen to satisfy \eqref{eq:int-pmp}. When this is possible, we arrive at \eqref{eq:adjoints} and \eqref{eq:max-princ} through Propositions \ref{prop:pull-back-flow} and \ref{prop:int-pmp}, respectively.

In order to carry out this program, we introduce a new type of directional derivative. Recall that
the \emph{lower Dini derivative} of a function $\ell : E \rightarrow \R$ is defined by
\begin{equation*}
D\ell(x;v) : = \inf_{\stackrel{v_n \rightarrow v}{ \lambda_n \downarrow 0}} \lim_{n \rightarrow \infty} \frac{\ell(x + \lambda_n v_n) - \ell(x)}{\lambda_n}.
\end{equation*}
The lower Dini derivative has played an important role in the study of nonsmooth functions since its appearance in the late $19^{th}$ century. More recently, a concept of \emph{weak lower Dini derivative} was introduced in \cite{cl1994}, the difference being that the sequences $v_n$ may converge to $v$ only weakly. We introduce the following intermediate notion:
\begin{definition}
Let $V \subset E$ be a closed, bounded, convex set and $v \in V$ be given. The \emph{set-constrained weak lower Dini derivative in direction $v$} is given by
\begin{equation*}
D^w_V \ell(x;v) : = \inf \lim_{n \rightarrow \infty} \frac{\ell(x + \lambda_n v_n) - \ell(x)}{\lambda_n}
\end{equation*}
where the infimum is over sequences $v_n \in V$ converging weakly to $v$ and $\lambda_n \downarrow 0$.
\end{definition}

\begin{remark}
  Since this construction is entirely local, we may define $D^w_V \ell(q;v)$ for a function $\ell : M \rightarrow \R$ and closed, bounded, convex set $V \subset T_qM$ using local coordinates.
\end{remark}

For locally Lipschitz function $\ell$, lower bounds on $D^w_V \ell(q;v)$ can be attained in terms of the limiting subgradient through the following Subbotin type result (see also \cite{clarkeand1998})
\begin{proposition}
\label{prop:lower-bound}
  Suppose that $\ell : M \rightarrow \R$ is locally Lipschitz rank $k_\ell$ and $V \subset T_qM$ is closed, bounded, and convex. For any $r$ such that $D^{w}_V \ell(q;v) \ge r$ for all $v \in V$ there exists $\zeta \in \pL \ell(q)$ such that $\left<\zeta, v\right> \ge r$ for all $v \in V$.
\end{proposition}

\begin{proof}
It suffices to prove the result for the case $M = E$.
Let $\lambda_n \downarrow 0$ be given and for a fixed integer $n$ consider the set $x + \lambda_n V$. By the multidirectional mean value inequality \cite{cl1994,zhu1998clarke}, there exists $z_n \in x + \lambda_n V + \lambda_n \B$ and $\zeta_n \in \pF \ell(z_n)$ such that
\begin{equation*}
 \inf_{w \in V, e \in \B} \ell(x + \lambda_n w + \lambda^2_n e) - \ell(x) - \lambda^2_n \le \lambda_n\left<\zeta_n, v\right>
\end{equation*}
for all $v \in V$. Choose $w_n \in V$ and $e_n \in \B$ such that
\begin{equation*}
  \ell(x + \lambda_n w_n + \lambda^2_n e_n) \le \inf_{w \in V, e \in \B} \ell(x + \lambda_n w + \lambda^2_n e) + \lambda^2_n.
\end{equation*}
 Then we have, for all $v \in V$,
\begin{equation*}
    \ell(x + \lambda_n w_n + \lambda^2_n e_n) - \ell(x) - 2 \lambda_n^2 \le \lambda_n \left< \zeta_n, v\right>.
\end{equation*}
Dividing by $\lambda_n > 0$ and using the Lipschitz property of $\ell$ we find that
\begin{equation*}
  \frac{\ell(x + \lambda_n w_n) - \ell(x)}{\lambda_n} - 2 \lambda_n - k_\ell \lambda_n \le \left<\zeta_n, v\right>
\end{equation*}
for any $v \in V$. Since $\ell$ is locally Lipschitz the $\zeta_n$ are bounded in norm we may pass to a subsequence such that $\zeta_n \weakly \zeta \in \pL \ell(x)$ and $w_n \weakly w \in V$. Taking the limit we find
\begin{equation*}
  r \le D^{w}_V \ell(x;w) \le \lim_{n \rightarrow \infty}  \frac{\ell(x + \lambda_n w_n) - \ell(x)}{\lambda_n} - 2 \lambda_n - k_\ell \lambda_n \le \left<\zeta, v\right>
\end{equation*}
and this completes the proof.
\end{proof}

For a given pair $(q_0, u) \in \cX^\circ$, we are interested in a set $V$ arising from infinitesimal perturbations $(v_0, v_1) \in T_{q_0}M \times T_{q(T;q_0,u)}M$ defined through
\begin{equation}
\label{eq:v0v1-source}
  (v_0, v_1) = \left. \frac{d}{d \lambda} \right|_{\lambda = 0} \left(q^\lambda(0), q^\lambda(T)\right)
\end{equation}
where $q^\lambda$ corresponds to some $v_0 \in \cB_{q_0}$ and $\nu \in \cM$ and is defined by \eqref{eq:basic-variation}. It is shown in \cite{kl2014extension} that if $(v_0, v_1)$ are defined through \eqref{eq:v0v1-source} then
\begin{equation*}
v_1 =  Q_{0,T \, *}(q_0)v_0 + \int_0^T Q_{t, T \, *}(q(t;q_0,u)) \left(\what{f}(t, q(t;q_0,u),\nu(t)) - \what{f}(t, q(t;q_0,u),\delta_{u(t)}) \right) \, dt,
\end{equation*}
where $Q_{s,t}$ denotes the flow of vector field $(t,q) \mapsto \what{f}(t, q, \delta_{u(t)})$. We note that because of \eqref{eq:integrable-bdd}, the set $V(q_0, u)$ is bounded and certainly it is convex.

\begin{lemma}
\label{lem:uniform-bound}
There is a constant $c_3$ such that for any $(q,u) \in \cX^\circ$, $(v_0, v_1) \in V(q,u)$, the variation $q^\lambda$ defined by \eqref{eq:basic-variation} satisfies
  \begin{equation}
  \label{eq:difference-quotient-approx}
    \left\|\psi_T(q^\lambda(T)) - \psi_T(q^0(T)) - \lambda \psi_{T \, *} v_1 \right\|_E \le c_3 \lambda^2
  \end{equation}
  when $\lambda$ is small enough that $q^\lambda(t) \in P_{0,t} (\cO_0)$ for all $t$.
\end{lemma}

\begin{proof}
  It is enough to prove that under Assumption $(D)$ we have the estimate
  \begin{equation}
  \label{eq:difference-quotient-approx-g}
  \left\|x^\lambda(T) - x(T) - \lambda \left. \frac{d x^\lambda}{d\lambda} (T)\right|_{\lambda = 0} \right\|_E \le c_3 \lambda^2,
  \end{equation}
  where $x^\lambda(t):= x(t;x_0 + \lambda \varphi_*(q) v_0, \mu^\lambda)$, $\mu^\lambda(t) := (1 - \lambda) \delta_{u(t)} + \lambda \nu(t)$, and $x(t) := x(t;x_0,u)$. We leave the details to the reader.

\end{proof}

Finally, the following lemma provides us with a useful relationship between the set-constrained lower derivative and variations.
\begin{lemma}
  \label{lem:variations-dwv}
  Let $\ell : M \times M \rightarrow \R$ be locally Lipschitz, $(q,u) \in \cX^\circ$, and $V := c \ell V(q,u)$. For any $(v_0, v_1) \in V$, there exists a sequence of variations $q^\lambda_n$ of the form \eqref{eq:basic-variation} and a sequence $\lambda_n \downarrow 0$ such that
  \begin{equation*}
    \lim_{n \rightarrow \infty} \frac{ \ell(q^{\lambda_n}_n(0),q^{\lambda_n}_n(T)) - \ell(q,q(T;q,u))}{\lambda_n} = D^w_V \ell(q,q(T;q,u); v_0, v_1).
  \end{equation*}
\end{lemma}
\begin{proof}
Let $(v_0^n, v_1^n) \weakly (v_0, v_1)$ and $\lambda_n \downarrow 0$ be chosen so that
\begin{equation}
\label{eq:a-bit-of-a-mess}
 \begin{aligned}
 & D^w_V \ell(q,q(T;q,u); v_0, v_1) \\
  & \hspace{.5cm} = \lim_{n \rightarrow \infty} \frac{\wtilde{\ell}\left(\psi_0(q) + \lambda_n \psi_{0 \, *} v_0^n, \psi_T(q(T;q,u)) + \lambda_n \psi_{T \, *}v_1^n\right) - \wtilde{\ell}\left(\psi_0(q), \psi_T(q(T;q,u))\right)}{\lambda_n},
 \end{aligned}
\end{equation}
where $\wtilde{\ell} := \ell \circ \left(\psi_0 \times \psi_T\right)^{-1}$. Without loss of generality we may assume that $(v_0^n, v_1^n) \in V(q,u)$. Let $q^\lambda_n$ be a variation of type \eqref{eq:basic-variation} corresponding to $(v_0^n, v_1^n)$. By Lemma \ref{lem:uniform-bound} we can write \eqref{eq:a-bit-of-a-mess} as
\begin{equation*}
 \begin{aligned}
 & D^w_V \ell(q,q(T;q,u); v_0, v_1) \\
  & \hspace{.5cm} = \lim_{n \rightarrow \infty} \frac{\wtilde{\ell}\left(\psi_0(q) + \lambda_n \psi_{0 \, *} v_0^n, \psi_T(q^{\lambda_n}_n(T))\right) - \wtilde{\ell}\left(\psi_0(q), \psi_T(q(T;q,u))\right)}{\lambda_n} \\
  & \hspace{.5cm} = \lim_{n \rightarrow \infty} \frac{\ell\left(q^{\lambda_n}_n(0),q^{\lambda_n}_n(T)\right) - \ell(q,q(T;q,u))}{\lambda_n}.\end{aligned}
\end{equation*}
\end{proof}

\emph{Proof of Theorem \ref{thm:pmp} for $S = M \times M$:}
We assume that $(q_0, u_*)$ is optimal and write $q_*(t) := q(t;q_0, u_*)$. Let $V:=c \ell \, V(q_0, u_*)$. Since $(q_0, u_*)$ is optimal, Lemma \ref{lem:variations-dwv} implies that $D^w_V \ell(q_*(0), q_*(T); v_0, v_1) \ge 0$ for all $(v_0, v_1) \in V$.

Proposition \ref{prop:lower-bound} now implies the existence of $(\xi, \eta) \in \pL \ell(q_0, q(T;q_0, u_*))$ such that for all $(v_0, v_1) \in V(q_0, u_*)$ we have
\begin{equation}
\label{eq:main-variational-tool}
  \left<\xi, v_0\right> + \left<\eta, v_1\right> \ge 0.
\end{equation}
In the definition of $v_1$, we may take $\nu = \delta_{u_*(t)}$ and so find that for all $v_0 \in \cB_{q_0}$, $\left<\xi, v_0\right> + \left<\eta, P_{0,T \, *} v_0\right> \ge 0$.
It follows that $\xi = -P_{0, T}^* \eta$. Setting $\zeta(t) := - P_{t, T}^*\eta$ we find that $(\zeta(0), -\zeta(T)) \in \pL \ell(q_0, q(T;q_0, u_*))$. Moreover, taking $v_0 = 0$ in \eqref{eq:main-variational-tool} we arrive at \eqref{eq:int-pmp} and this completes the proof. \qed

\section{Controllability and Metric Regularity} \label{section:metric-regularity}

In the case where the endpoints are constrained, we cannot take quite the same approach and instead use a metric regularity argument based on the following form of the Ekeland principle:

\begin{proposition}[Ekeland]
\label{prop:ekeland}
Let $(\cX,\rho)$ be a complete pseudometric space and let $\ell : \cX \rightarrow \ol{\R}$ be lower semicontinuous. Suppose that $\ol{x}$ satisfies $\ell(\ol{x}) < \inf_{x \in \cX} \ell(x) + \ve$ for some $\ve > 0$. For any $\sigma > 0$ there exists $y \in \cX$ such that $\rho(\ol{x},y) < \sigma$ and which is a global minimizer of the perturbed function
\begin{equation}
  z \mapsto \ell(z) + \frac{\ve}{\sigma} \rho(y,z).
\end{equation}
\end{proposition}

Proposition \ref{prop:ekeland} was introduced for complete metric spaces in \cite{ekeland1974variational} and it was shown in \cite{ledyaev-doc-diss} that the principle holds for complete pseudometric spaces. The paper \cite{borwein1986} contains a proof in the setting of metric spaces which can be adopted to establish the above statement.

For this paper, of course, $(\cX, \rho)$ is the pseudometric space defined in Section \ref{pseudometric-subsection}.
It can be shown, using pointwise a.e. convergence of subsequences in $L^1$ and existence of measurable selections complete, separable metric space $\U$, that the space $(\cX, \rho)$ is a complete pseudometric space.

Metric regularity of constraints \eqref{eq:dynamic-constraint} and \eqref{eq:state-constraint} is closely related to a concept of controllability. To describe this concept, we need a notion of distance toward the set $S \subset M \times M$ and it is convenient to work with a function $\Phi : \cO \times P_{0,T}(\cO) \rightarrow \R$ by
\begin{equation*}
  \Phi(q, q^\prime) = d_{(\psi_0 \times \psi_T)(S)}\left(\psi_0(q), \psi_T(q^\prime)\right).
\end{equation*}
By a result of \cite{ledyaevzhu2007}, one may check that $\pL \Phi(q, q^\prime) \subset N_S^L(q, q^\prime)$ when $(q,q^\prime) \in S$.

\begin{definition}
\label{defn:weakly-controllable}
  A pair $(q_0, u_*)$ is said to be \emph{weakly controllable in the direction of $S$} if there exist $\ve_0, \Delta > 0$ such that for any pair $(q_1, u)$ satisfying $\rho(q_0,u_*; q_1, u) < \ve_0$ there exists $(v_0, v_1) \in V(q_1,u)$ such that
  \begin{equation}
  \label{eq:strong-decrease}
  D^w_{c \ell \, V(q_1,u)} \Phi(q_1, q(T;q_1, u) ; v_0, v_1) < - \Delta.
  \end{equation}
\end{definition}

We write $\cA \subset \cX$ for the set of pairs $(q_0, u)$ which satisfy $(q_0, q(T;q_0, u)) \in S$.

\begin{proposition}
Suppose that $(q_0, u_*) \in \cA$, $S$ is sequentially normally compact in a neighborhood of $(q_0, q(T;q_0, u_*))$, and that $(q_0, u_*)$ is not weakly controllable in the direction of $S$. Then Theorem \ref{thm:pmp} holds with $\lambda_0 = 0$.
\end{proposition}

\begin{proof}
For any $\ve_n, \Delta_n \downarrow 0$ there exists a sequence of pairs $(q_n,u_n)$ with $\rho(q_0,u_*; q_n,u_n) < \ve_n$ and $(q_n,q(T;q_n, u_n)) \not \in S$ such that for all $(v_0, v_1) \in V(q_n, u_n)$ we have
\begin{equation*}
  D^w_{c \ell \, V(q_n,u_n)} \Phi(q_n,q(T;q_n, u_n); v_0,v_1) \ge - \Delta_n
\end{equation*}
for all $n$. Consequently there exist $(\xi_n,\eta_n) \in \pL \Phi(q_n,q(T;q_n, u_n))$ such that for all $(v_0, v_1) \in V(q_n,u_n)$ we have
\begin{equation*}
  \left<\xi_n, v_0\right> + \left<\eta_n, v_1\right> \ge -\Delta_n.
\end{equation*}
Since $S$ is sequentially normally compact we may assume without loss of generality that $(\xi_n,\eta_n)$ are nonzero. Since $E$ is reflexive, we may pass to the limit and obtain nonzero $(\xi, \eta) \in N_S^L(q_0, q(T;q_0, u_*))$ such that
\begin{equation*}
  \left<\xi, v_0\right> + \left<\eta, v_1\right> \ge 0
\end{equation*}
for all $(v_0, v_1) \in V(q_0, u_*)$. Following the same line of reasoning as in the paragraph following \eqref{eq:main-variational-tool}, we find that $\zeta(t) := -Q_{t,T}^*\eta$ satisfies \eqref{eq:adjoints} and \eqref{eq:max-princ}, where $Q_{s,t}$ is the flow of $(t,q) \mapsto f(t, q, u_*(t))$. Since $Q_{s,t}^*$ is an isomorphism, we see that $\zeta(t) \ne 0$ for all $t$, completing the proof.
\end{proof}

Next we consider the case in which an admissible pair is weakly controllable toward $S$. Let
\begin{equation*}
  \rho_{\cA}(q,u) :=   \inf \left\{ \rho(q, u; q^\prime, w) \, : \, (q^\prime,w) \in \cA \right\}
\end{equation*}
\begin{proposition}
Suppose that $(q_0, u_*) \in \cA$ is weakly controllable in the direction of $S$. Then there exist constants $\ve_1 > 0$ and $\kappa > 0$ such that for any $(q, u)$ with $\rho(q_0, u_*; q, u) < \ve_1$ we have
\begin{equation}
\label{eq:metric-regularity}
\rho_{\cA}(q,u) \le \kappa \Phi(q, q(T;q, u)).
\end{equation}
\end{proposition}
\begin{proof}
Since $(q_0, u_*)$ is weakly controllable toward $S$ we may choose $\ve_0, \Delta > 0$ such that for any $(q,u)$ with $\rho(q_0, u_*; q, u) < \ve_0$ we have \eqref{eq:strong-decrease}.
Set $\kappa := 4 c_1 / \Delta$ and choose $0 < \ve_1 < \ve_0/2$ small enough that for all $\rho(q_0, u_*; q,u) < \ve_1$ we have $(q,u) \in \cX^\circ$ and $\kappa \Phi(q,q(T;q,u)) < \ve_0 / 2$.

Suppose by way of contradiction that for some pair $(\ol{q}, \ol{u})$ satisfying $\rho(q_0, u_*; \ol{q},\ol{u}) < \ve_1$ we have
  \begin{equation*}
\rho_{\cA}(\ol{q}, \ol{u}) > \kappa \Phi(\ol{q}, q(T;\ol{q}, \ol{u})).
  \end{equation*}
  Necessarily, $(\ol{q}, \ol{u}) \not \in \cA$.
Applying Proposition \ref{prop:ekeland} with
  \begin{equation*}
    \sigma := \kappa \Phi(\ol{q}, q(T;\ol{q}, \ol{u})) \hspace{1cm} \ve := 2 \Phi(\ol{q}, q(T;\ol{q}, \ol{u}))
  \end{equation*}
  we obtain $(q,u) \in \cX$ with $\rho(q,u; \ol{q}, \ol{u}) < \kappa \Phi(\ol{q}, q(T;\ol{q}, \ol{u}))$ for which the perturbed functional
  \begin{equation*}
    (q^\prime,w) \mapsto \Phi(q^\prime,q(T;q^\prime,w)) + \frac{2}{\kappa} \rho(q,u;q^\prime,w)
  \end{equation*}
  is minimized. Check that $(q,u) \in \cX^\circ$ and $(q,u) \not \in \cA$. As a consequence we may choose $(v_0, v_1) \in  V(q,u)$ such that
\begin{equation*}
  D^w_V(q, q(T;q,u);v_0, v_1) \le - \Delta
\end{equation*}
and by Lemma \ref{lem:variations-dwv} we may choose a sequence of variations $q_n^\lambda$ and a sequence $\lambda_n\downarrow 0$ with
\begin{equation}
\label{eq:strong-decrease-phi}
  \lim_{n \rightarrow \infty} \frac{\Phi(q^{\lambda_n}_n(0),q^{\lambda_n}_n(T)) - \Phi(q, q(T;q,u))}{\lambda_n} \le - \Delta.
\end{equation}
By Lemma \ref{lem:pseudo-approx} we can assume that $q^{\lambda_n}$ corresponds to usual control $w^{\lambda_n}$ satisfying
\begin{equation*}
  \rho(q^{\lambda_n}_n(0), w^{\lambda_n} ; q, u) \le c_1 \lambda + c_2 \lambda^2.
\end{equation*}
Now
\begin{equation*}
    \Phi(q, q(T;q,u)) \le \Phi(q^{\lambda_n}_n(0),q^{\lambda_n}_n(T)) + \frac{2}{\kappa} c_1 \lambda_n  +  \frac{2}{\kappa} c_2 \lambda_n^2
\end{equation*}
and so \eqref{eq:strong-decrease-phi} implies
\begin{equation*}
  0 \le - \Delta + \frac{2}{\kappa}c_1 = - \frac{\Delta}{2} < 0.
\end{equation*}
This contradiction proves that \eqref{eq:metric-regularity} holds for $\rho(q_0,u_*; q,u) < \ve_1$ and so completes the proof.
\end{proof}

A standard exact penalization argument (see e.g. \cite{kl2014}) now implies that if $(q_0, u_*)$ is weakly controllable in the direction of $S$ then it is locally optimal (in $\rho$ pseudometric) for the problem of minimizing a cost
\begin{equation*}
  J(q,u) := \ell(q, q(T;q,u)) + K \Phi(q, q(T;q,u))
\end{equation*}
when $K$ is sufficiently large. Here constraint \eqref{eq:state-constraint} has been removed through exact penalization. This proves Theorem \ref{thm:penalization} and since we have shown Theorem \ref{thm:pmp} to be true for the free endpoints problem we now see that if $(q_0, u_*)$ is weakly controllable in the direction of $S$, then Theorem \ref{thm:pmp} holds with $\lambda_0 = 1$.

\end{document}